\documentclass[12pt,a4paper]{article}
\usepackage[T1]{fontenc}
\usepackage[utf8]{inputenc}
\usepackage{lmodern}
\usepackage[margin=27mm]{geometry}
\usepackage{amsmath,amssymb,amsthm}
\usepackage{mathtools}
\usepackage{microtype}
\usepackage{enumitem}
\usepackage{cite}
\usepackage[hidelinks,breaklinks]{hyperref}
\usepackage{fancyhdr}
\usepackage{needspace}
\newtheorem{theorem}{Theorem}[section]
\newtheorem{lemma}[theorem]{Lemma}
\newtheorem{proposition}[theorem]{Proposition}

\newtheorem{conjecture}[theorem]{Conjecture}
\theoremstyle{remark}

\newcommand{\doi}[1]{\href{https://doi.org/#1}{\nolinkurl{#1}}}

\title{\bfseries Tuza's Conjecture under an Edge-Count Condition}
\author{Zhen Liu\footnote{Email: 1552580575@qq.com}, ~Qinghou Zeng\footnote{Research supported by National Key R\&D Program of China (Grant No. 2023YFA1010202) and National Natural Science Foundation of China (Grant No. 12371342). Email: zengqh@fzu.edu.cn (Corresponding
author)}\\
{\small Center for Discrete Mathematics, Fuzhou University, Fujian, 350003, China}}
\date{}

\begin{document}
\maketitle
\thispagestyle{plain}

\begin{abstract}
For a graph $G$, let $\nu(G)$ denote the maximum number of pairwise edge-disjoint triangles and let $\tau(G)$ denote the minimum number of edges meeting every triangle. Tuza's conjecture asserts that $\tau(G)\le 2\nu(G)$ for every graph $G$. We first note that Tuza's conjecture is equivalent to the following statement: for every fixed $c\ge 0$, if $G$ is a sufficiently large $n$-vertex graph with $e(G)\ge n^2/4-cn$, then $\tau(G)\le 2\nu(G)$. In this paper, we prove this statement for every $c<2$.
\end{abstract}



\section{Introduction}
All graphs considered in this paper are finite, simple, and undirected. A \emph{triangle} is a cycle of length three. A \emph{triangle packing} is a family of pairwise edge-disjoint triangles, while a \emph{triangle cover} is a set of edges meeting every triangle. We denote by $\nu(G)$ the maximum size of a triangle packing in $G$, and by $\tau(G)$ the minimum size of a triangle cover. The union of the edges of a maximum triangle packing is a triangle cover, and hence
$
\nu(G)\le \tau(G)\le 3\nu(G).
$
In 1981, Tuza proposed the following conjecture; see~\cite{Tuz84,Tuz90}.

\begin{conjecture}[Tuza]\label{conj:tuza}
Every graph $G$ satisfies
$
\tau(G)\le 2\nu(G).
$
\end{conjecture}

Since $\nu(K_4)=1$ and $\tau(K_4)=2$, the coefficient $2$ in Tuza's conjecture is best possible. The conjecture remains open in general; the best known universal bounds are due to Haxell~\cite{Hax99}, Yi~\cite{Yi26}, and Wang~\cite{Wan26}, who proved respectively the coefficients $66/23$, $63/22$, and $165/59$.
Tuza's conjecture has been studied extensively for various special classes of graphs and through several related fractional formulations; see~\cite{Tuz90,ABT12,Pul15,BFG21,BBG22,KP22,Kri95,HKT12,CDM14}.

For dense graphs, several more precise results are known. 
Tuza~\cite{Tuz90} proved that Conjecture~\ref{conj:tuza} holds for every 
$n$-vertex graph $G$ with
$
e(G)\ge \frac{7}{16}n^2.
$
A \emph{fractional triangle packing} assigns a nonnegative weight $x_T$ 
to each triangle $T$ such that
$
\sum_{T\ni e}x_T\le 1
$
for every edge $e$. The maximum value of $\sum_T x_T$ is denoted by 
$\nu^*(G)$. Krivelevich~\cite{Kri95} proved that
$
\tau(G)\le 2\nu^*(G),
$
and Haxell and R\"odl~\cite{HR01} proved that
$
\nu^*(G)-\nu(G)=o(n^2).
$
Consequently,
$
\tau(G)\le 2\nu(G)+o(n^2).
$

Yuster~\cite{Yus12} proved that for every $\beta>0$ there are $n_0=n_0(\beta)$ and $\delta=\delta(\beta)$ such that, if a graph $G$ has $n>n_0$ vertices and $m\ge \beta n^2$ edges with $\tau(G)\ge (1-\delta)\frac{m}{2},$ then $\nu(G)\ge \left(1+\frac{\beta^2}{800}\right)\frac{m}{4}.$
Baron and Kahn~\cite{BK16} proved that for every $\alpha>0$ there are graphs $G$ of arbitrarily large order and positive density such that
$\tau(G)>(1-o(1))e(G)/2$ and
$\nu(G)<(1+\alpha)e(G)/4$.
Thus the ratio $\tau(G)/\nu(G)$ can be arbitrarily close to $2$ even for dense graphs.
More recently, Chahua and Guti\'errez~\cite{CG25} proved Tuza's conjecture for split graphs of minimum degree at least $3n/5$, proved
$\tau(G)<28\nu(G)/15$
for tripartite graphs of minimum degree greater than $33n/56$, and proved
$\tau(G)\le 3\nu(G)/2$
for complete $4$-partite graphs on at least five vertices.

The following simple observation gives a useful equivalent form of Tuza's conjecture.
\begin{proposition}\label{prop:equivalent}
Tuza's conjecture is equivalent to: for every fixed $c\ge 0$, every sufficiently large $n$-vertex graph $G$ with $e(G)\ge n^2/4-cn$ satisfies $\tau(G)\le 2\nu(G)$.
\end{proposition}

\begin{proof}
The forward implication is immediate.

For the converse, assume the stated condition holds for every fixed $c\ge 0$. Let $F$ be an arbitrary graph on $h$ vertices. For a positive integer $r$, let $G=F\cup K_{r,r}$ be the disjoint union. Since $K_{r,r}$ is triangle-free, every triangle of $G$ lies in $F$, and hence $\nu(G)=\nu(F)$ and $\tau(G)=\tau(F)$.

Now $G$ has $n=h+2r$ vertices and $e(G)=e(F)+r^2$ edges. Therefore
\[
\frac{n^2}{4}-e(G)
=\frac{(h+2r)^2}{4}-e(F)-r^2
\le rh+\frac{h^2}{4}
\le \frac{h}{2}(h+2r)
=\frac{h}{2}n,
\]
where the last inequality follows from $h^2/4\le h^2/2$. Thus $e(G)\ge n^2/4-(h/2)n$. Taking $c=h/2$, which is fixed, and choosing $r$ sufficiently large, the assumed statement applies to $G$. Hence $\tau(G)\le 2\nu(G)$, and therefore $\tau(F)\le 2\nu(F)$. Since $F$ was arbitrary, Tuza's conjecture follows.
\end{proof}

Our main result is the following.

\begin{theorem}\label{thm:main}
For every fixed real number $c<2$, every sufficiently large $n$-vertex graph $G$ with
$
e(G)\ge \frac{n^2}{4}-cn
$
satisfies
$
\tau(G)\le2\nu(G).
$
\end{theorem}

\noindent
\textbf{Notation.}
Throughout the paper, all graphs are finite, simple, and undirected. For a graph $G$, we write $V(G)$ and $E(G)$ for its vertex and edge sets, and $e(G)=|E(G)|$. For $v\in V(G)$, let $N_G(v)$ and $d_G(v)$ denote its neighborhood and degree. For $U\subseteq V(G)$, let $G[U]$ denote the subgraph of $G$ induced by $U$, and write $G-U=G[V(G)\setminus U]$. If $F\subseteq E(G)$, then $G-F$ denotes the graph obtained from $G$ by deleting the edges in $F$. For $U\subseteq V(G)$, let $e_G(U)=e(G[U])$, and for disjoint sets $U,W\subseteq V(G)$, let $e_G(U,W)$ denote the number of edges with one endpoint in $U$ and the other in $W$. If $S\subseteq E(G)$, then $e_S(U)$ and $e_S(U,W)$ have the analogous meanings with only the edges of $S$ counted.

\section{Key lemmas}\label{sec:key}

In this section, we present several lemmas that are useful in the proof of our main result.

Recall that a \emph{fractional triangle packing} of a graph $G$ is a function $x$ from the set of triangles of $G$ to $[0,1]$ such that $\sum_{T\ni e}x(T)\le 1$ for every $e\in E(G)$. The maximum value of $\sum_T x(T)$ is denoted by $\nu^*(G)$. A \emph{fractional triangle cover} of a graph $G$ is a function $f:E(G)\to\mathbb R_{\ge0}$ such that $\sum_{e\in E(T)}f(e)\ge1$ for every triangle $T$ of $G$. The minimum value of $\sum_{e\in E(G)}f(e)$ over all fractional triangle covers of $G$ is denoted by $\tau^{*}(G)$. By linear programming duality, $\tau^{*}(G)=\nu^{*}(G)$; see~\cite[Section~2]{Yus12}.

\begin{lemma}[Haxell--R\"odl~\cite{HR01}]\label{lem:HR}
For every $\eta>0$ and all sufficiently large $n$, every $n$-vertex
graph $G$ satisfies $0\le\nu^*(G)-\nu(G)\le\eta n^2$.
\end{lemma}

A graph $G$ is \emph{$4$-colorable} if there exists a map
$\phi:V(G)\to\{1,2,3,4\}$ such that
$\phi(x)\ne\phi(y)$ for every $xy\in E(G)$.

\begin{lemma}[Aparna Lakshmanan--Bujt\'as--Tuza~\cite{ABT12}]
\label{lem:four-colorable}
Every $4$-colorable graph $G$ satisfies
$\tau(G)\le2\nu(G)$.
\end{lemma}

A graph is \emph{tripartite} if its vertex set can be partitioned
into three independent sets.

\begin{theorem}[Haxell--Kohayakawa~\cite{HK98}]
\label{thm:tripartite}
There is a constant $\varepsilon>0.044$ such that every tripartite
graph $G$ satisfies $\tau(G)\le(2-\varepsilon)\nu(G)$.
\end{theorem}

An edge set $S\subseteq E(G)$ is \emph{triangle-independent} if no
triangle of $G$ contains two edges of $S$.

\begin{lemma}\label{lem:triangle-independent}
Let $G$ be a graph on $n\ge1$ vertices and let $S\subseteq E(G)$
be triangle-independent. Put $C=E(G)\setminus S$.
If $|C|\ge\beta n^2$, where $0<\beta < 1/2$, then
$\tau(G)\le |C|-\beta^3n^2/36$.
\end{lemma}

\begin{proof}
For disjoint sets $A,B\subseteq V(G)$, let $H$ be
the spanning subgraph of $G$ with
\[
E(H)=\bigl(S\setminus(E_G(A)\cup E_G(B))\bigr)
\cup\bigl(C\cap E_G(A,B)\bigr).
\]
Then
\[
E(H)\cap C=C\cap E_G(A,B),
\qquad E_H(A)=E_H(B)=\varnothing.
\]

Since $S$ is triangle-independent, every triangle of $H$
contains at least two edges of $C$, both joining $A$ and $B$.
The third edge must then lie in $E_H(A)\cup E_H(B)=\varnothing$,
a contradiction. Hence $H$ is triangle-free, and therefore
\begin{equation}\label{eq:key-retain}
\tau(G)\le |E(G)\setminus E(H)|
=|C|-e_C(A,B)+e_S(A)+e_S(B).
\end{equation}

Let $G_S=(V(G),S)$ and set $W_1=V(G)$.
Suppose that $W_i$ has been defined.
If there exists a vertex $v_i\in W_i$ with
$d_{G_S[W_i]}(v_i)\ge\beta n/2$,
let $A_i=N_{G_S[W_i]}(v_i)$ and
$W_{i+1}=W_i\setminus A_i$.
Otherwise, stop the process.
Suppose that the process produces $A_1,\ldots,A_m$,
and let $W=W_{m+1}$.
Since the sets $A_i$ are pairwise disjoint and
$|A_i|\ge\beta n/2$, we have $m\le2/\beta$.
Moreover, every $A_i$ is independent in $G$.
Indeed, an edge $xy\in E_G(A_i)$ would form a triangle
$v_ixy$ with $v_ix,v_iy\in S$, contradicting the
triangle-independence of $S$.

By the stopping rule,
$d_{G_S[W]}(v)<\beta n/2$ for every $v\in W$.
Hence
\[
e_S(W)=\frac12\sum_{v\in W}d_{G_S[W]}(v)
<\frac{\beta n|W|}{4}
\le\frac{\beta n^2}{4}.
\]
Since $V(G)=A_1\cup\cdots\cup A_m\cup W$
and each $A_i$ is independent in $G$, we have
\[
|C|=\sum_{1\le i<j\le m}e_C(A_i,A_j)
+\sum_{i=1}^m e_C(A_i,W)+e_C(W).
\]
As $|C|\ge\beta n^2$, at least one of
$\sum_{1\le i<j\le m}e_C(A_i,A_j)$,
$\sum_{i=1}^m e_C(A_i,W)$, and $e_C(W)$
is at least $\beta n^2/3$.

We now consider the three cases above.
If
$
\sum_{1\le i<j\le m} e_C(A_i,A_j)\ge \frac{\beta n^2}{3},
$
then, since $\binom{m}{2}\le m^2/2\le2/\beta^2$, there exist
$1\le i<j\le m$ such that
$
e_C(A_i,A_j)\ge\frac{\beta^3}{6}n^2.
$
Since both $A_i$ and $A_j$ are independent in $G$,
taking $A=A_i$ and $B=A_j$ in~\eqref{eq:key-retain}
gives the desired bound.

If $\sum_{i=1}^m e_C(A_i,W)\ge\beta n^2/3$,
then $m\le2/\beta$ implies that some $1\le i\le m$
satisfies $e_C(A_i,W)\ge\beta^2n^2/6$.
Set $A=A_i$, and choose $B\subseteq W$ by including
each vertex independently with probability $\beta/3$.
Since $A$ is independent in $G$, we have $e_S(A)=0$.
By linearity of expectation,
\[
\begin{aligned}
\mathbb E[e_C(A,B)-e_S(A)-e_S(B)]
&=\frac{\beta}{3}e_C(A,W)
-\frac{\beta^2}{9}e_S(W)\\
&\ge\frac{\beta}{3}\cdot\frac{\beta^2n^2}{6}
-\frac{\beta^2}{9}\cdot\frac{\beta n^2}{4}\\
&=\frac{\beta^3n^2}{36}.
\end{aligned}
\]
Thus some choice of $B$ satisfies
$e_C(A,B)-e_S(A)-e_S(B)\ge\beta^3n^2/36$.
The desired bound follows from~\eqref{eq:key-retain}.

Finally, suppose that $e_C(W)\ge \beta n^2/3$. Partition $W$ randomly into $A$ and $B$, assigning each vertex independently to either part with probability $1/2$. Then $\mathbb E\,e_C(A,B)=e_C(W)/2$ and $\mathbb E[e_S(A)+e_S(B)]=e_S(W)/2$. Hence
\[
\begin{aligned}
\mathbb E\bigl[e_C(A,B)-e_S(A)-e_S(B)\bigr]
=\frac12\bigl(e_C(W)-e_S(W)\bigr)\ge \frac{\beta n^2}{24}
\ge \frac{\beta^3 n^2}{36},
\end{aligned}
\]
where the last inequality follows from $0<\beta<1/2$. Therefore some partition $W=A\cup B$ gives the required bound in~\eqref{eq:key-retain}.
\end{proof}

\begin{lemma}\label{lem:peeling}
Let $G$ be a graph and let $uv\in E(G)$. If
$|N_G(u)\cap N_G(v)|>2\nu(G)$, then
$\nu(G-uv)=\nu(G)-1$ and
$\tau(G)\le1+\tau(G-uv)$.
\end{lemma}

\begin{proof}
Let $\mathcal P$ be a maximum triangle packing of $G-uv$.
Since $uv\notin E(G-uv)$, every triangle of $\mathcal P$
contains at most one of $u$ and $v$.
Thus the triangles of $\mathcal P$ use at most
$2|\mathcal P|=2\nu(G-uv)\le2\nu(G)$
edges incident with $u$ or $v$.
Since $|N_G(u)\cap N_G(v)|>2\nu(G)$,
there exists $w\in N_G(u)\cap N_G(v)$ such that
neither $uw$ nor $vw$ belongs to any triangle of $\mathcal P$.
Hence $\mathcal P\cup\{uvw\}$ is a triangle packing of $G$,
giving $\nu(G)\ge\nu(G-uv)+1$.
Conversely, deleting the triangle containing $uv$, if any,
from a maximum triangle packing of $G$ gives
$\nu(G-uv)\ge\nu(G)-1$.
Therefore $\nu(G-uv)=\nu(G)-1$.
Finally, adding $uv$ to a minimum triangle cover of $G-uv$
gives a triangle cover of $G$.
Hence $\tau(G)\le1+\tau(G-uv)$.
\end{proof}

A \emph{vertex cover} of a graph $G$ is a set $X\subseteq V(G)$ such that every edge of $G$ has at least one endpoint in $X$.

\begin{lemma}\label{lem:witness}
Let $V_1$ be a vertex cover of a graph $G$. Then there exists
$V_2$ with $V_1\subseteq V_2\subseteq V(G)$ and
$|V_2|\le |V_1|^2+2|V_1|$ such that the minimum size
of a vertex cover of $G[V_2]$ equals that of $G$.
\end{lemma}

\begin{proof}
For each $x\in V_1$, choose a subset
$U_x\subseteq N_G(x)\setminus V_1$ with
$|U_x|=\min\{|N_G(x)\setminus V_1|,|V_1|+1\}$,
and set $V_2=V_1\cup\bigcup_{x\in V_1}U_x$.
Note that the sets $U_x$ need not be pairwise disjoint.
Nevertheless,
\[
|V_2|\le |V_1|+\sum_{x\in V_1}|U_x|
\le |V_1|^2+2|V_1|.
\]

Let $X$ be a minimum vertex cover of $G[V_2]$.
Since $V_1$ is a vertex cover of $G[V_2]$,
we have $|X|\le |V_1|$.
Suppose that $X$ is not a vertex cover of $G$.
Then there exists an edge $xy\in E(G)$ with
$x,y\notin X$.
Since $V_1$ is a vertex cover of $G$, we may assume
that $x\in V_1\subseteq V_2$.
As $X$ covers every edge of $G[V_2]$, we have
$y\notin V_2$.
Thus $y\in N_G(x)\setminus(V_1\cup U_x)$,
which implies $|U_x|=|V_1|+1$.
Since $x\notin X$ and $X$ covers every edge
of $G[V_2]$, we have $U_x\subseteq X$.
Hence $|X|\ge |V_1|+1$, a contradiction.
Therefore $X$ is also a vertex cover of $G$.
Since $G[V_2]$ is an induced subgraph of $G$,
the two graphs have the same minimum vertex-cover size.
\end{proof}

\begin{lemma}\label{lem:three-vertices}
Let $G$ be a graph. If $G-S$ is bipartite for some
$S\subseteq V(G)$ with $|S|\le3$, then
$\tau(G)\le2\nu(G)$.
\end{lemma}

\begin{proof}
Write $V(G)\setminus S=X\cup Y$, where $X$ and $Y$
are independent in $G$.
If $G[S]$ is bipartite, then $G$ is $4$-colorable,
and the result follows from Lemma~\ref{lem:four-colorable}.
Otherwise, $|S|\le3$ implies that
$S=\{a,b,c\}$, where $T=abc$ is a triangle.
Let $H=G-E(T)$.
Since $S,X,Y$ are independent in $H$, the graph $H$
is tripartite.
Moreover, $\nu(G)\ge\nu(H)+1$ and
$\tau(G)\le\tau(H)+3$.

If $\nu(H)>0$, Theorem~\ref{thm:tripartite} gives
$\tau(H)<2\nu(H)$.
Since both parameters are integers, we obtain
\[
\tau(G)\le\tau(H)+3
\le2\nu(H)+2
\le2\nu(G).
\]
Now suppose that $\nu(H)=0$.
Then every triangle of $G$ contains an edge of $T$,
so $\tau(G)\le3$.
Thus we may assume that $\nu(G)=1$.
If two edges of $T$ form a triangle cover of $G$,
the result follows.
Otherwise, there exist triangles $abx$, $bcy$,
and $caz$, where $x,y,z\notin S$.
Since $\nu(G)=1$, any two triangles of $G$
share an edge, which forces $x=y=z=d$.
Thus $G$ contains a copy of $K_4$ on
$\{a,b,c,d\}$.
Any triangle containing a vertex outside this set
is edge-disjoint from at least one of the four
triangles of this $K_4$, contradicting $\nu(G)=1$.
Hence every triangle of $G$ lies in this $K_4$,
and $\{ad,bc\}$ is a triangle cover of $G$.
Therefore $\tau(G)\le2=2\nu(G)$.
\end{proof}

\section{Proof of the main theorem}\label{sec:main}

In this section, we prove Theorem~\ref{thm:main}.
Fix $c<2$, and let $G$ be an $n$-vertex graph with
$e(G)\ge n^2/4-cn$, where $n$ is sufficiently large.
In particular, we may assume that $n^{2/3}<10^{-5}n^2$.
If $\nu(G)=0$, then $G$ is triangle-free and
$\tau(G)=0$. Thus we may assume $\nu(G)>0$
and consider the following three cases.

\medskip
\noindent\textbf{Case 1: $\nu(G)\ge10^{-5}n^2$.}

Suppose, for a contradiction, that $\tau(G)>2\nu(G)$.
By Lemma~\ref{lem:HR}, we may assume that
\begin{equation}\label{eq:main-fractional-gap}
0\le\nu^*(G)-\nu(G)\le10^{-18}n^2.
\end{equation}

Let $f:E(G)\to[0,1]$ be a minimum fractional
triangle cover of $G$. By linear programming duality,
$\sum_{e\in E(G)}f(e)=\nu^*(G)$.
For each $s\in[0,1/3]$, let
$Z_s=\{e\in E(G):f(e)<s\}$ and
$S_s=\{e\in E(G):f(e)\ge(1-s)/2\}$.
Since $s\le(1-s)/2$, the sets $Z_s$ and $S_s$
are disjoint. Moreover, every triangle containing
an edge of $Z_s$ must also contain an edge of $S_s$.
Otherwise, such a triangle $T$ would satisfy
$\sum_{e\in E(T)}f(e)<s+2(1-s)/2=1$,
contrary to the covering condition.

Choose a partition $V(G)=V_1\cup V_2$ maximizing
$|E_G(V_1,V_2)\setminus(S_s\cup Z_s)|$.
By the standard maximum-cut bound, we have
\[
|E_G(V_1,V_2)\setminus(S_s\cup Z_s)|
\ge\frac12|E(G)\setminus(S_s\cup Z_s)|.
\]

Let $H$ be the spanning subgraph of $G$ with
\[
E(H)=Z_s\cup
\bigl(E_G(V_1,V_2)\setminus(S_s\cup Z_s)\bigr).
\]
Then $H$ is triangle-free. Indeed, any triangle
of $H$ containing an edge of $Z_s$ must also
contain an edge of $S_s$, whereas
$E(H)\cap S_s=\varnothing$.
Any other triangle of $H$ would have all its
edges between $V_1$ and $V_2$, which is impossible.
Since
$e(H)\ge |Z_s|+(e(G)-|S_s|-|Z_s|)/2$,
deleting $E(G)\setminus E(H)$ gives
\begin{equation}\label{tauG}
\tau(G)\le e(G)-e(H)
\le\frac{e(G)+|S_s|-|Z_s|}{2}.
\end{equation}

By the standard maximum-cut bound, we also have
$2\tau(G)\le e(G)$.
Order the edges as $e_1,\ldots,e_{e(G)}$ such that
$f(e_1)\ge\cdots\ge f(e_{e(G)})$.
Since every triangle contains an edge $e$ with
$f(e)\ge1/3$, the set
$\{e\in E(G):f(e)\ge1/3\}$ is a triangle cover.
Thus $f(e_{\tau(G)})\ge1/3$.

We claim that
\begin{equation}\label{f_i}
2f(e_i)+f(e_{2\tau(G)+1-i})\ge1
\qquad(1\le i\le\tau(G)).
\end{equation}
Otherwise, there exists $1\le j\le\tau(G)$ such that
$2f(e_j)+f(e_{2\tau(G)+1-j})<1$.
Since $f(e_j)\ge1/3$, we can choose
$s\in[0,1/3]$ satisfying
$f(e_{2\tau(G)+1-j})<s<1-2f(e_j)$.
The inequality $f(e_j)<(1-s)/2$ implies that
$|S_s|\le j-1$. Similarly,
$f(e_{2\tau(G)+1-j})<s$ implies that
$e(G)-|Z_s|\le2\tau(G)-j$.
By~\eqref{tauG},
\[
\tau(G)\le
\frac{(2\tau(G)-j)+(j-1)}{2}
=\tau(G)-\frac12,
\]
a contradiction. This proves~\eqref{f_i}.

Summing~\eqref{f_i} over $i=1,\ldots,\tau(G)$
and using $\sum_{e\in E(G)}f(e)=\nu^*(G)$,
we obtain
\[
\begin{aligned}
\nu^*(G)
&= \sum_{i=1}^{\tau(G)}f(e_i) +\sum_{i=\tau(G)+1}^{e(G)}f(e_i)\ge \frac{\tau(G)}{2}-\frac12\sum_{i=\tau(G)+1}^{2\tau(G)}f(e_i)+\sum_{i=\tau(G)+1}^{e(G)}f(e_i)\\
&=\frac{\tau(G)}{2}+\frac12\sum_{i=\tau(G)+1}^{2\tau(G)}f(e_i)+\sum_{i=2\tau(G)+1}^{e(G)}f(e_i)\ge \frac{\tau(G)}{2}+\frac12\sum_{i=\tau(G)+1}^{e(G)}f(e_i).
\end{aligned}
\]
Thus 
\begin{align}\label{sumtau}
\sum_{i=\tau(G)+1}^{e(G)}f(e_i)\le2(\nu^*(G)-\tau(G)/2). 
\end{align}
For each $1\le i\le\tau(G)$, we also have
\[
\min\{f(e_i),1/2\}
+\min\{f(e_{2\tau(G)+1-i}),1/2\}
\ge\frac12.
\]
Indeed, this is immediate if $f(e_i)\ge1/2$.
Otherwise, both weights are less than $1/2$,
and the inequality follows from~\eqref{f_i}.
Summing over $i=1,\ldots,\tau(G)$ gives
$\sum_{e\in E(G)}\min\{f(e),1/2\}\ge\tau(G)/2$.
Consequently,
\begin{equation}\label{sumhalf}
\begin{aligned}
\sum_{\substack{e\in E(G)\\f(e)>1/2}}
\left(f(e)-\frac12\right)
=\nu^*(G)-\sum_{e\in E(G)}
\min\{f(e),1/2\}\le\nu^*(G)-\frac{\tau(G)}{2}.
\end{aligned}
\end{equation}

Since $\tau(G)>2\nu(G)$,
\eqref{eq:main-fractional-gap} and~\eqref{sumhalf}
give
$0\le\nu^*(G)-\tau(G)/2<10^{-18}n^2$.

Set $L=\{e\in E(G):f(e)<1/6\}$,
$M=\{e\in E(G):1/6\le f(e)<2/3\}$,
and $F=\{e\in E(G):f(e)\ge2/3\}$.
Since $f(e_{\tau(G)})\ge1/3$, the first
$\tau(G)$ edges all belong to $F\cup M$.
Every remaining edge of $F\cup M$ has weight
at least $1/6$. Hence, by~\eqref{sumtau},
\[
\begin{aligned}
|F|+|M|-\tau(G)\le6\sum_{i=\tau(G)+1}^{e(G)}f(e_i)\le12\left(\nu^*(G)-\frac{\tau(G)}{2}\right)
<12\cdot10^{-18}n^2.
\end{aligned}
\]

Moreover, every edge $e\in F$ satisfies
$f(e)-1/2\ge1/6$. By~\eqref{sumhalf},
$|F|\le6(\nu^*(G)-\tau(G)/2)
<6\cdot10^{-18}n^2$.
Therefore
\[
|M|\ge\tau(G)-|F|
>2\nu(G)-|F|> 2\cdot10^{-5}n^2-6\cdot10^{-18}n^2
>10^{-5}n^2.
\]

We claim that $L$ is triangle-independent in $G-F$.
Otherwise, some triangle $T$ of $G-F$ contains
two edges of $L$, giving
$\sum_{e\in E(T)}f(e)<1/6+1/6+2/3=1$,
a contradiction.
Since $E(G-F)\setminus L=M$ and
$|M|>10^{-5}n^2$, applying
Lemma~\ref{lem:triangle-independent} with
$\beta=10^{-5}$ gives
$\tau(G-F)\le |M|-10^{-15}n^2/36$.
Hence
\[
\begin{aligned}
\tau(G)
&\le |F|+\tau(G-F)\le |F|+|M|-\frac{10^{-15}}{36}n^2\\
&<\tau(G)+
\left(12\cdot10^{-18}
-\frac{10^{-15}}{36}\right)n^2<\tau(G),
\end{aligned}
\]
a contradiction. This proves Case 1.

We now establish a decomposition of dense triangle-free graphs
that will be used in Cases 2 and 3.
Let $H$ be an $n$-vertex triangle-free graph with
$e(H)\ge(1/4-10^{-4})n^2$.
Choose a vertex $v$ of maximum degree $d$ in $H$, and set
$V_1=N_H(v)$ and $V_2=V(H)\setminus V_1$.
Since $H$ is triangle-free, $V_1$ is independent.
Moreover,
\[
e(H)+e_H(V_2)
=\sum_{u\in V_2}d_H(u)
\le d(n-d)\le\frac{n^2}{4}.
\]
Thus $e_H(V_2)\le n^2/4-e(H)$ and
$e_H(V_1,V_2)\ge2e(H)-n^2/4$.
Consequently,
\begin{equation}\label{|V_1|-eV_1}
\begin{aligned}
|V_1||V_2|-e_H(V_1,V_2)\le\frac{n^2}{4}-e_H(V_1,V_2)\le2\left(\frac{n^2}{4}-e(H)\right)
\le2\cdot10^{-4}n^2.
\end{aligned}
\end{equation}

Since $|V_1|+|V_2|=n$,
$|V_1||V_2|=n^2/4-(|V_1|-n/2)^2$.
Therefore
\begin{align}\label{V_1-n/2}
\left(|V_1|-\frac n2\right)^2
&=\frac{n^2}{4}-|V_1||V_2|
\le\frac{n^2}{4}-e_H(V_1,V_2)
\le2\left(\frac{n^2}{4}-e(H)\right)
\le2\cdot10^{-4}n^2.
\end{align}
Since $|V_2|-n/2=-(|V_1|-n/2)$,
\eqref{V_1-n/2} gives
\[
\left||V_1|-\frac n2\right|
=\left||V_2|-\frac n2\right|
\le\frac{\sqrt2\,n}{100}.
\]

Define
\[
\begin{aligned}
A={}&\{x\in V_1:|V_2\setminus N_H(x)|>n/5\}\cup\{x\in V_2:|V_1\setminus N_H(x)|>n/5\},
\end{aligned}
\]
and put $B=V_1\setminus A$ and $C=V_2\setminus A$.
By the definition of $A$, we have
\[
\begin{aligned}
\frac n5|A|
&\le \sum_{x\in A\cap V_1}|V_2\setminus N_H(x)|
+\sum_{x\in A\cap V_2}|V_1\setminus N_H(x)|\\
&\le \sum_{x\in V_1}|V_2\setminus N_H(x)|
+\sum_{x\in V_2}|V_1\setminus N_H(x)|\\
&=2\bigl(|V_1||V_2|-e_H(V_1,V_2)\bigr).
\end{aligned}
\]
Hence, by~\eqref{|V_1|-eV_1},
\begin{equation}\label{eq:main-A-bound}
|A|\le\frac{10}{n}
\bigl(|V_1||V_2|-e_H(V_1,V_2)\bigr)
\le\frac n{500}.
\end{equation}

Since $B=V_1\setminus A$ and $C=V_2\setminus A$,
\eqref{V_1-n/2} and~\eqref{eq:main-A-bound} give
\begin{equation}\label{eq:main-BC-balance}
\left||B|-\frac n2\right|,
\quad
\left||C|-\frac n2\right|
\le\left(\frac{\sqrt2}{100}+\frac1{500}\right)n.
\end{equation}
In particular, $|B|,|C|>9n/20$.

Since $B\subseteq V_1$ and $C\subseteq V_2$,
\eqref{|V_1|-eV_1} gives
\begin{equation}\label{eq:main-BC-missing}
\begin{aligned}
|B||C|-e_H(B,C)
&\le |V_1||V_2|-e_H(V_1,V_2)\le 2\left(\frac{n^2}{4}-e(H)\right)\le 2\cdot10^{-4}n^2.
\end{aligned}
\end{equation}

By the definition of $A$, we have
\[
\begin{aligned}
|C\setminus N_H(x)|
&\le |V_2\setminus N_H(x)|\le n/5
&& (x\in B),\\
|B\setminus N_H(x)|
&\le |V_1\setminus N_H(x)|\le n/5
&& (x\in C).
\end{aligned}
\]

Since $|B|,|C|>9n/20$, we obtain
\begin{equation}\label{eq:main-common}
\begin{aligned}
|N_H(x)\cap N_H(y)\cap C|
&\ge |C|-\frac{2n}{5}>\frac n{20}
&& (x,y\in B,\ x\ne y),\\
|N_H(x)\cap N_H(y)\cap B|
&\ge |B|-\frac{2n}{5}>\frac n{20}
&& (x,y\in C,\ x\ne y).
\end{aligned}
\end{equation}

Finally, $B$ is independent in $H$ since
$B\subseteq V_1$.
If $xy\in E_H(C)$, then~\eqref{eq:main-common}
implies that $x$ and $y$ have a common neighbor
$z\in B$. Thus $xyz$ is a triangle in $H$,
a contradiction. Therefore
\begin{equation}\label{BCindep}
E_H(B)=E_H(C)=\varnothing.
\end{equation}

\medskip
\noindent\textbf{Case 2: $n^{2/3}<\nu(G)<10^{-5}n^2$.}

Let $\mathcal P$ be a maximum triangle packing of $G$,
and let $H=G-\bigcup_{T\in\mathcal P}E(T)$.
By the maximality of $\mathcal P$, the graph $H$ is
triangle-free. Moreover,
\[
e(H)=e(G)-3\nu(G)
\ge\frac{n^2}{4}-cn-3\nu(G)
\ge\left(\frac14-10^{-4}\right)n^2,
\]
where the last inequality holds for sufficiently large
$n$, since $c<2$ and $\nu(G)<10^{-5}n^2$ imply
$cn+3\nu(G)<2n+3\cdot10^{-5}n^2\le10^{-4}n^2$.
We also have
\begin{equation}\label{eq:main-H-deficit}
\frac{n^2}{4}-e(H)
=\frac{n^2}{4}-e(G)+3\nu(G)
\le cn+3\nu(G).
\end{equation}
Apply the preceding construction to $H$ to obtain
$V_1,V_2,A,B,C$ satisfying
\eqref{eq:main-A-bound}, \eqref{eq:main-BC-missing},
\eqref{eq:main-common}, and~\eqref{BCindep}.

For each $x\in A$, define
\[
\begin{aligned}
L_x&=(N_H(x)\cap B)
\cup\{b\in B:xbb'\in\mathcal P
\text{ for some }b'\in B\},\\
R_x&=(N_H(x)\cap C)
\cup\{c\in C:xcc'\in\mathcal P
\text{ for some }c'\in C\}.
\end{aligned}
\]
We first claim that
\begin{equation}\label{eq:main-noLR}
e_H(L_x,R_x)=0
\qquad(x\in A).
\end{equation}

Suppose, to the contrary, that $bc\in E(H)$
for some $b\in L_x$ and $c\in R_x$.
If $xb,xc\in E(H)$, then $xbc$ is a triangle
of $H$, a contradiction.
Suppose next that exactly one of $xb,xc$
belongs to $E(H)$.
By symmetry, assume that $xb\in E(H)$.
Since $c\in R_x$ and $xc\notin E(H)$,
there exists $c'\in C$ such that
$xcc'\in\mathcal P$.
By~\eqref{eq:main-common}, we may choose
$z\in N_H(c)\cap N_H(c')\cap B$ with $z\ne b$.
Then
\[
(\mathcal P\setminus\{xcc'\})\cup\{xbc,zcc'\}
\]
is a triangle packing of $G$.
Indeed, the edges $xc,cc'$ belong to the removed
triangle, while $xb,bc,zc,zc'\in E(H)$.
Since $z\ne b$, the two new triangles are
edge-disjoint, and neither shares an edge with
any triangle of $\mathcal P\setminus\{xcc'\}$.
This gives a packing of size $|\mathcal P|+1$,
a contradiction.
It remains to consider $xb,xc\notin E(H)$.
By the definitions of $L_x$ and $R_x$,
there exist $b'\in B$ and $c'\in C$ such that
$xbb',xcc'\in\mathcal P$.
These two triangles are distinct since
$B\cap C=\varnothing$.
By~\eqref{eq:main-common}, we may choose
\[
\begin{aligned}
c_0&\in(N_H(b)\cap N_H(b')\cap C)
\setminus\{c,c'\},\\
b_0&\in(N_H(c)\cap N_H(c')\cap B)
\setminus\{b,b'\}.
\end{aligned}
\]
Consider
\[
(\mathcal P\setminus\{xbb',xcc'\})
\cup\{xbc,bb'c_0,b_0cc'\}.
\]
The edges $xb,xc,bb',cc'$ belong to the two
removed triangles, while
$bc,bc_0,b'c_0,b_0c,b_0c'\in E(H)$.
Since $b_0\notin\{b,b'\}$ and
$c_0\notin\{c,c'\}$, the three new triangles
are pairwise edge-disjoint.
Moreover, they share no edge with any triangle
of $\mathcal P\setminus\{xbb',xcc'\}$.
Thus we obtain a triangle packing of size
$|\mathcal P|+1$, again a contradiction.
This proves~\eqref{eq:main-noLR}.

We next claim that
\begin{equation}\label{eq:main-no-BC-triangle}
\text{no triangle of }\mathcal P
\text{ is contained in }B\text{ or }C.
\end{equation}
By symmetry, suppose that
$b_1b_2b_3\in\mathcal P$ with
$b_1,b_2,b_3\in B$.
By~\eqref{eq:main-common}, we may choose
distinct vertices $c_1,c_2\in C$ such that
$c_1\in N_H(b_1)\cap N_H(b_2)$ and
$c_2\in N_H(b_2)\cap N_H(b_3)$.
Then
\[
(\mathcal P\setminus\{b_1b_2b_3\})
\cup\{b_1b_2c_1,b_2b_3c_2\}
\]
is a triangle packing of size $|\mathcal P|+1$.
Indeed, $b_1b_2,b_2b_3$ belong to the removed
triangle, while the other four edges belong
to $H$. The two new triangles are edge-disjoint
since $c_1\ne c_2$.
This contradicts the maximality of $\mathcal P$
and proves~\eqref{eq:main-no-BC-triangle}.

Set $A_1=\{x\in A:|L_x|\le|R_x|\}$ and
$A_2=A\setminus A_1$, and let
$X=B\cup A_1$ and $Y=C\cup A_2$.
Since $L_x\subseteq B$, $R_x\subseteq C$,
and $e_H(L_x,R_x)=0$, we have
$|L_x||R_x|\le|B||C|-e_H(B,C)$
for every $x\in A$.
Consequently,
\begin{equation}\label{eq:main-LR-sum}
\begin{aligned}
\sum_{x\in A_1}|L_x|+\sum_{x\in A_2}|R_x|
&=\sum_{x\in A}\min\{|L_x|,|R_x|\}\le\sum_{x\in A}\sqrt{|L_x||R_x|}\\
&\le |A|\sqrt{|B||C|-e_H(B,C)}.
\end{aligned}
\end{equation}

Let
\[
\mathcal Q=\bigl\{T\in\mathcal P:
(V(T)\subseteq X\text{ or }V(T)\subseteq Y),
\ |V(T)\cap A|\ge2\bigr\}.
\]
Each triangle of $\mathcal P$ contributes one edge
to $e_G(X)+e_G(Y)$, and two additional edges
if all its vertices lie in $X$ or $Y$.
By~\eqref{eq:main-no-BC-triangle}, every triangle
lying entirely in one part but not belonging to
$\mathcal Q$ has one of the forms
$xbb'$ with $x\in A_1$ and $b,b'\in B$,
or $xcc'$ with $x\in A_2$ and $c,c'\in C$.
Since $\mathcal P$ is edge-disjoint and
$E(H)\cap E(T)=\varnothing$ for every
$T\in\mathcal P$, the definitions of $L_x,R_x$
give
\[
\begin{aligned}
|L_x|-|N_H(x)\cap B|
&=2|\{xbb'\in\mathcal P:b,b'\in B\}|,\\
|R_x|-|N_H(x)\cap C|
&=2|\{xcc'\in\mathcal P:c,c'\in C\}|
\end{aligned}
\]
for every $x\in A$.

By~\eqref{BCindep}, we also have
\[
\begin{aligned}
e_H(X)+e_H(Y)=e_H(A_1)+e_H(A_2)+\sum_{x\in A_1}|N_H(x)\cap B|
+\sum_{x\in A_2}|N_H(x)\cap C|.
\end{aligned}
\]
Combining these identities and using
$|\mathcal P|=\nu(G)$, we obtain
\[
\begin{aligned}
e_G(X)+e_G(Y)
&=e_H(X)+e_H(Y)+\nu(G)+2|\mathcal Q|+\sum_{x\in A_1}
\bigl(|L_x|-|N_H(x)\cap B|\bigr)\\
&\quad+\sum_{x\in A_2}
\bigl(|R_x|-|N_H(x)\cap C|\bigr)\\
&=\nu(G)+e_H(A_1)+e_H(A_2)+2|\mathcal Q|+\sum_{x\in A_1}|L_x|
+\sum_{x\in A_2}|R_x|.
\end{aligned}
\]

Every triangle of $\mathcal Q$ contains an edge
of $G[A_1]$ or $G[A_2]$.
These edges are distinct for different triangles
of $\mathcal Q$ and do not belong to $H$.
Therefore
\[
e_H(A_1)+e_H(A_2)+|\mathcal Q|
\le e_G(A_1)+e_G(A_2)
\le\binom{|A|}{2}.
\]

Deleting all edges inside $X$ and $Y$
leaves a bipartite graph.
Hence, by~\eqref{eq:main-LR-sum},
\[
\begin{aligned}
\tau(G)
&\le e_G(X)+e_G(Y)\le\nu(G)+|A|\sqrt{|B||C|-e_H(B,C)}
+2\binom{|A|}{2}\\
&\le\nu(G)+|A|\sqrt{|B||C|-e_H(B,C)}
+|A|^2.
\end{aligned}
\]
Using~\eqref{eq:main-A-bound} and
\eqref{eq:main-BC-missing}, we obtain
\[
\begin{aligned}
\tau(G)
&\le\nu(G)
+\frac{10}{n}
\bigl(|V_1||V_2|-e_H(V_1,V_2)\bigr)^{3/2}+\frac{100}{n^2}
\bigl(|V_1||V_2|-e_H(V_1,V_2)\bigr)^2\\
&\le\nu(G)
+\frac{20\sqrt2}{n}
\left(\frac{n^2}{4}-e(H)\right)^{3/2}+\frac{400}{n^2}
\left(\frac{n^2}{4}-e(H)\right)^2\\
&\le\nu(G)
+\frac{20\sqrt2}{n}
\bigl(cn+3\nu(G)\bigr)^{3/2}+\frac{400}{n^2}
\bigl(cn+3\nu(G)\bigr)^2,
\end{aligned}
\]
where the second and third inequalities follow
from~\eqref{|V_1|-eV_1} and
\eqref{eq:main-H-deficit}, respectively.

Since $c<2$, we have
$cn+3\nu(G)<2n+3\nu(G)$.
Using
$(a+b)^{3/2}\le\sqrt2(a^{3/2}+b^{3/2})$
and $(a+b)^2\le2(a^2+b^2)$ for $a,b\ge0$,
we obtain
\[
\begin{aligned}
\frac{\tau(G)-\nu(G)}{\nu(G)}
&\le\frac{20\sqrt2}{n\nu(G)}
(2n+3\nu(G))^{3/2}+\frac{400}{n^2\nu(G)}
(2n+3\nu(G))^2\\
&\le80\sqrt2\,\frac{\sqrt n}{\nu(G)}+120\sqrt3\,\frac{\sqrt{\nu(G)}}{n}
+\frac{3200}{\nu(G)}+7200\frac{\nu(G)}{n^2}\\
&<80\sqrt2\,n^{-1/6}+120\sqrt{3\cdot10^{-5}}+3200n^{-2/3}
+7200\cdot10^{-5}\\
&<\frac34+120n^{-1/6}+3200n^{-2/3}
<1,
\end{aligned}
\]
where the last two inequalities follow from
$n^{2/3}<\nu(G)<10^{-5}n^2$
and sufficiently large $n$.
Therefore $\tau(G)<2\nu(G)$.
This proves Case 2.

\medskip
\noindent\textbf{Case 3: $\nu(G)\le n^{2/3}$.}

Choose a spanning subgraph $G'\subseteq G$ with the fewest
edges subject to
\begin{equation}\label{eq:main-minimal-subgraph}
e(G)-e(G')=\nu(G)-\nu(G').
\end{equation}
Such a graph exists since $G$ itself satisfies this equality.
Suppose that some $uv\in E(G')$ satisfies
$|N_{G'}(u)\cap N_{G'}(v)|>2\nu(G')$.
By Lemma~\ref{lem:peeling},
$\nu(G'-uv)=\nu(G')-1$.
Since $e(G'-uv)=e(G')-1$, the graph $G'-uv$
also satisfies~\eqref{eq:main-minimal-subgraph},
contradicting the choice of $G'$.
Therefore
\begin{equation}\label{eq:main-bounded-common}
|N_{G'}(u)\cap N_{G'}(v)|
\le2\nu(G')
\qquad(uv\in E(G')).
\end{equation}
A triangle cover of $G'$, together with
$E(G)\setminus E(G')$, covers every triangle of $G$.
Hence
\begin{equation}\label{eq:main-cover-transfer}
\tau(G)\le e(G)-e(G')+\tau(G')
=\nu(G)-\nu(G')+\tau(G').
\end{equation}
If $\nu(G')=0$, then $\tau(G')=0$ and
$\tau(G)\le\nu(G)$. Thus we may assume $\nu(G')>0$.

Let $H$ be obtained from $G'$ by deleting the edges
of a maximum triangle packing. By maximality, $H$ is triangle-free.
By~\eqref{eq:main-minimal-subgraph}, we have
\begin{equation}\label{eq:main-H-lower}
\begin{aligned}
e(H)
&=e(G')-3\nu(G')\\
&=e(G)-\nu(G)-2\nu(G')\\
&\ge\frac{n^2}{4}-cn-3\nu(G)
\ge\frac{n^2}{4}-5n,
\end{aligned}
\end{equation}
where the last inequality follows from $c<2$ and
$\nu(G')\le\nu(G)\le n^{2/3}\le n$.
Since $5n\le10^{-4}n^2$ for sufficiently large $n$,
we may apply the decomposition preceding Case 2 to $H$,
retaining the notation $V_1,V_2,A,B,C$.

We first claim that $B$ and $C$ are independent in $G'$.
Otherwise, there exist distinct vertices $u,v\in B$
or $u,v\in C$ with $uv\in E(G')$.
By~\eqref{eq:main-bounded-common} and
\eqref{eq:main-common},
\[
2\nu(G')
\ge |N_{G'}(u)\cap N_{G'}(v)|
\ge |N_H(u)\cap N_H(v)|
>\frac n{20}>2\nu(G'),
\]
where the last inequality holds for sufficiently large
$n$, since $\nu(G')\le n^{2/3}$.
This contradiction proves that
$e_{G'}(B)=e_{G'}(C)=0$.

Since $H$ is triangle-free, for every $x\in A$,
we have
$e_H(N_H(x)\cap B,N_H(x)\cap C)=0$.
Since $N_H(x)\cap B\subseteq B\subseteq V_1$
and $N_H(x)\cap C\subseteq C\subseteq V_2$,
we obtain
\[
\begin{aligned}
|N_H(x)\cap B|\,|N_H(x)\cap C|
&\le |B||C|-e_H(B,C)\le |V_1||V_2|-e_H(V_1,V_2)\\
&\le 2\left(\frac{n^2}{4}-e(H)\right)\le 10n.
\end{aligned}
\]
The last two inequalities follow from
\eqref{|V_1|-eV_1} and
\eqref{eq:main-H-lower}, respectively.

Set
$A_1=\{x\in A:|N_H(x)\cap B|\le|N_H(x)\cap C|\}$
and $A_2=A\setminus A_1$.
Let $X=B\cup A_1$ and $Y=C\cup A_2$.
By~\eqref{eq:main-A-bound} and
\eqref{|V_1|-eV_1},
\[
|A|\le\frac{20}{n}
\left(\frac{n^2}{4}-e(H)\right)\le100,
\]
where the last inequality follows from
\eqref{eq:main-H-lower}.
Since $B$ and $C$ are independent in $H$,
we have
\[
\begin{aligned}
e_H(X)+e_H(Y)
&=e_H(A_1)+e_H(A_2)
+e_H(A_1,B)+e_H(A_2,C)\\
&=e_H(A_1)+e_H(A_2)+\sum_{x\in A}
\min\{|N_H(x)\cap B|,|N_H(x)\cap C|\}\\
&\le\binom{|A|}{2}
+\sum_{x\in A}
\sqrt{|N_H(x)\cap B|\,|N_H(x)\cap C|}\\
&\le\binom{|A|}{2}+|A|\sqrt{10n}\le\binom{100}{2}+100\sqrt{10n}.
\end{aligned}
\]

Since $V_1=B\cup(A\cap V_1)$ and $X=B\cup A_1$,
we have $|X|-|V_1|=|A_1|-|A\cap V_1|$.
Similarly, $|Y|-|V_2|=|A_2|-|A\cap V_2|$.
Hence
$\bigl||X|-|V_1|\bigr|,
\bigl||Y|-|V_2|\bigr|\le|A|\le100$.
On the other hand, by~\eqref{V_1-n/2} and
\eqref{eq:main-H-lower}, we have
\[
\left||V_1|-\frac n2\right|
=\left||V_2|-\frac n2\right|
\le\sqrt{2\left(\frac{n^2}{4}-e(H)\right)}
\le\sqrt{10n}.
\]
Since $H$ is obtained from $G'$ by deleting
$3\nu(G')$ edges, the preceding estimates give
\begin{equation}\label{eq:main-sublinear-estimates}
\begin{aligned}
\left||X|-\frac n2\right|,
\quad\left||Y|-\frac n2\right|
&\le\sqrt{10n}+100,\\
e_{G'}(X)+e_{G'}(Y)
&\le e_H(X)+e_H(Y)+3\nu(G')\\
&\le\binom{100}{2}+100\sqrt{10n}
+3\nu(G').
\end{aligned}
\end{equation}

Let $S_1,S_2$ be minimum vertex covers of
$G'[X]$ and $G'[Y]$, respectively.
Since $B,C$ are independent in $G'$,
the sets $A_1,A_2$ are vertex covers of
$G'[X],G'[Y]$, respectively. Hence
$|S_1|+|S_2|\le|A_1|+|A_2|=|A|\le100$.

By Lemma~\ref{lem:witness}, there exist
$S_1\subseteq T_1\subseteq X$ and
$S_2\subseteq T_2\subseteq Y$
such that the minimum vertex-cover size of
$G'[T_i]$ equals $|S_i|$ for $i=1,2$.
Moreover,
\[
\begin{aligned}
|T_1|+|T_2|\le|S_1|^2+|S_2|^2
+2(|S_1|+|S_2|)\le10200.
\end{aligned}
\]

For each $y\in Y\setminus T_2$, the set
$T_1\setminus N_{G'}(y)$ covers every edge
of $G'[T_1]$ except those contained in
$G'[N_{G'}(y)\cap T_1]$.
Adding one endpoint of each remaining edge
gives a vertex cover of $G'[T_1]$.
Since its minimum vertex-cover size is $|S_1|$,
we have
\[
|T_1\setminus N_{G'}(y)|
\ge |S_1|-e_{G'}(N_{G'}(y)\cap T_1).
\]
Summing over $y\in Y\setminus T_2$ and
using~\eqref{eq:main-bounded-common} gives
\[
\begin{aligned}
\sum_{y\in Y\setminus T_2}
e_{G'}(N_{G'}(y)\cap T_1)&=\sum_{uv\in E(G'[T_1])}
|N_{G'}(u)\cap N_{G'}(v)\cap(Y\setminus T_2)|\\
&\le2\nu(G')e_{G'}(T_1).
\end{aligned}
\]
Consequently,
\[
\begin{aligned}
\sum_{y\in Y\setminus T_2}
|T_1\setminus N_{G'}(y)|&\ge \sum_{y\in Y\setminus T_2} (|S_1|-e_{G'}(N_{G'}(y)\cap T_1)) \\&\ge |S_1|(|Y|-|T_2|)-2\nu(G')e_{G'}(T_1).
\end{aligned}
\]
By symmetry,
\[
\begin{aligned}
\sum_{x\in X\setminus T_1}
|T_2\setminus N_{G'}(x)|\ge |S_2|(|X|-|T_1|)-2\nu(G')e_{G'}(T_2).
\end{aligned}
\]

Since $T_1\subseteq X$ and $T_2\subseteq Y$,
the two sums below count distinct nonadjacent
vertex pairs between $X$ and $Y$. Hence
\begin{equation}\label{eq:main-missing}
\begin{aligned}
|X||Y|-e_{G'}(X,Y)
&\ge\sum_{y\in Y\setminus T_2}
|T_1\setminus N_{G'}(y)|+\sum_{x\in X\setminus T_1}
|T_2\setminus N_{G'}(x)|\\
&\ge |S_1|(|Y|-|T_2|)
+|S_2|(|X|-|T_1|)-2\nu(G')
\bigl(e_{G'}(T_1)+e_{G'}(T_2)\bigr)\\
&\ge\frac{|S_1|+|S_2|}{2}n
-O(\sqrt n+\nu(G')).
\end{aligned}
\end{equation}

The last inequality follows from
\eqref{eq:main-sublinear-estimates},
$|S_1|+|S_2|\le100$, and
$|T_1|+|T_2|\le10200$.

Combining~\eqref{eq:main-sublinear-estimates}
and~\eqref{eq:main-missing}, we obtain
\[
\begin{aligned}
e(G')
&=e_{G'}(X,Y)+e_{G'}(X)+e_{G'}(Y)\\
&\le |X||Y|
-\frac{|S_1|+|S_2|}{2}n
+O(\sqrt n+\nu(G'))\\
&\le\frac{n^2}{4}
-\frac{|S_1|+|S_2|}{2}n
+O(\sqrt n+\nu(G')),
\end{aligned}
\]
where the last inequality follows from $|X|+|Y|=n$.
On the other hand, \eqref{eq:main-minimal-subgraph}
gives
\[
e(G')=e(G)-\nu(G)+\nu(G')
\ge\frac{n^2}{4}-cn-\nu(G)+\nu(G').
\]
Comparing these bounds yields
\[
\begin{aligned}
\left(\frac{|S_1|+|S_2|}{2}-c\right)n
&\le\nu(G)-\nu(G')
+O(\sqrt n+\nu(G'))=O(n^{2/3}),
\end{aligned}
\]
since $\nu(G')\le\nu(G)\le n^{2/3}$.
If $|S_1|+|S_2|\ge4$, then
$(2-c)n\le O(n^{2/3})$, which is impossible
for sufficiently large $n$, since $c<2$ is fixed.
Hence $|S_1|+|S_2|\le3$.

Since $S_1,S_2$ cover all edges of
$G'[X],G'[Y]$, respectively,
the graph $G'-(S_1\cup S_2)$ is bipartite.
By Lemma~\ref{lem:three-vertices},
$\tau(G')\le2\nu(G')$.
Finally, \eqref{eq:main-cover-transfer} gives
\[
\begin{aligned}
\tau(G)\le\nu(G)-\nu(G')+\tau(G')\le\nu(G)+\nu(G')\le2\nu(G).
\end{aligned}
\]
This proves Case 3 and completes the proof
of Theorem~\ref{thm:main}.

\section*{Declaration on the Use of Generative AI}
The authors used ChatGPT 6 Pro to assist in discussing proof strategies, checking proofs, and
improving exposition.

\end{document}